\documentclass{amsart}
\usepackage{amsthm,amssymb}

\usepackage{hyperref}
\usepackage[all]{xy}
\usepackage{amsmath,amsfonts}	
\usepackage{amsthm,latexsym,amssymb}

\newtheorem{thm}{Theorem}
\newtheorem{prop}[thm]{Proposition}
\newtheorem{lem}[thm]{Lemma}
\newtheorem{cor}[thm]{Corollary}

\theoremstyle{definition}
\newtheorem{defn}[thm]{Definition}
\newtheorem{ex}[thm]{Example}
\newtheorem{rem}[thm]{Remark}

\renewcommand{\[}{\begin{equation*}}
\renewcommand{\]}{\end{equation*}}

\def\frakX{\mathfrak{X}}
\def\frakL{\mathfrak{L}}

\def\scrF{\mathcal{F}}
\def\scrK{\mathcal{K}}
\def\T{\mathrm{T}}
\let\oldc=\c
\def\c{\mathrm{c}}
\def\G{\mathbb{G}}
\def\Con{\mathfrak{C}\mathrm{on}}
\DeclareMathOperator\Lie{Lie}

\DeclareMathOperator\id{id}
\DeclareMathOperator\End{End}

\begin{document}
\title[Extremal K-contact metrics ]{Extremal K-contact metrics}
\author{Mehdi Lejmi}
\address{D\'epartement de Math\'ematiques, Universit\'e Libre de Bruxelles CP218, Boulevard du Triomphe, Bruxelles 1050, Belgique. }
\email{mlejmi@ulb.ac.be}
\author{Markus Upmeier}
\address{D\'epartement de Math\'ematiques, Universit\'e Libre de Bruxelles CP218, Boulevard du Triomphe, Bruxelles 1050, Belgique. }
\email{mupmeier@ulb.ac.be}

\begin{abstract}
Extending a result of He to the non-integrable case of $K$-contact manifolds, it is shown that transverse Hermitian scalar curvature may be interpreted as a moment map for
the strict contactomorphism group. As a consequence, we may generalize the Sasaki-Futaki invariant to $K$-contact geometry and establish a number of elementary properties.

Moreover, we prove that in dimension $5$ certain deformation-theoretic results can be established also under weaker integrability conditions by exploiting the relationship between $J$-anti-invariant and self-dual $2$-forms.
\end{abstract}

\maketitle
\section{Introduction}


On a symplectic manifold $(M,\omega)$, consider the space $\mathcal{AC}(\omega)$ of all $\omega$-compatible almost-complex structures $J$ and the subspace $\mathcal{C}(\omega)$ of integrable ones.
A crucial observation due to Fujiki~\cite{fuj} is that $\mathcal{C}(\omega)$ may be viewed as an infinite-dimensional K\"ahler manifold and that the natural action of the group of Hamiltonian symplectomorphisms admits a moment map, associating to a complex structure $J$ the scalar curvature of the metric $g=\omega(\cdot, J\cdot)$. An important generalization of this result to $\mathcal{AC}(\omega)$, the non-integrable case, was established by Donaldson~\cite{don}.

The critical points of the square-norm of this moment map give canonical representatives of almost-complex structures $J$ (corresponding to metrics) called {\it{extremal almost-K\"ahler metrics}}~\cite{apo-dra,lej-1}. These metrics are a natural extension of Calabi's extremal K\"ahler metrics~\cite{cal-1,cal-2}.

Recently, He~\cite{he} introduced a similar moment map picture to Sasakian geometry, which may be viewed as an odd-dimensional counterpart of K\"ahler geometry. The first goal of this paper is to generalize in Theorem~\ref{moment_map} the result of He to the non-integrable case (so-called \emph{$K$-contact structures}), as conjectured in~\cite[Remark 4.3]{he}. The moment map now takes a $K$-contact structure to its transverse Hermitian scalar curvature. 

We define \emph{extremal $K$-contact metrics} again as critical points. Theorem~\ref{moment_map} has a number of consequences (such as a $K$-contact \emph{Futaki invariant}~\cite{bo-gal-sim,bo-gal-sim-1,fut}), which we investigate in Sections~\ref{sec4},~\ref{sec5}. These metrics appear as natural extensions of {\it{extremal Sasakian metrics}}, introduced by Boyer--Galicki--Simanca~\cite{bo-gal-sim,bo-gal-sim-1} and motivated by the examples of {\it{irregular}} Sasaki-Einstein metrics (see for instance~\cite{gau-mar-spa-wal}).

In Sections~\ref{sec3},~\ref{deform_section} we consider the deformation-theoretic behaviour of extremal $K$-contact metrics, leading to the notion of a \emph{semi-Sasakian structure}. As opposed to the integrable case, our considerations are limited to dimension $5$, as we exploit the relationship between $J$-anti-invariant and self-dual $2$-forms. We also generalize the {\it{transverse $\partial\overline\partial$-Lemma}} \cite{kac} to the $K$-contact case.
\medskip

{\bf Acknowledgements: } The first author is very grateful to Christina T{\o}nnesen-Friedman and Charles Boyer for their suggestions on how to construct irregular extremal Sasakian metrics. Both authors are thankful to Joel Fine and Weiyong He for several useful discussions.

\section{Preliminaries}

\subsection{$K$-contact structures}

Let $(M,\eta)$ be a {\it{contact manifold}} of dimension $2n+1$, where $\eta$ is the \emph{contact $1$-form} satisfying $\eta\wedge(d\eta)^n\neq 0$ at every point of $M$. The Reeb vector field $\xi \in \frakX(M)$ for $\eta$ is uniquely determined by the requirements
\begin{equation*}
\eta(\xi)=1,\quad \iota_\xi\left(d\eta\right)=0.
 \end{equation*}
The corresponding distribution $\scrF_\xi = \mathbb{R}\xi \subset TM$ defines the \emph{characteristic foliation}.

Denote by $\Con(M,\eta)$ the strict contactomorphism group of all diffeomorphisms $f$ satisfying $f^*\eta=\eta$. Its Lie algebra are all vector fields $X$ with $\frakL_X \eta = 0$. For $M$ compact,
the \emph{contact Hamiltonian} of $X$ is the unique basic function $f \in C^\infty_{B}(M)$ satisfying $\eta(X)=f$, $d\eta(X,\cdot) = -df$. This gives an identification
\begin{equation}\label{contact-hamiltonian}
	\Lie\Con(M,\eta) \cong C^\infty_{B}(M),\qquad f\leftrightarrow X_f.
\end{equation}
We use it to transport the metric on $C^\infty_{B}(M)\subset L^2(M, dv_\eta)$ to $\Lie \Con(M,\eta)$.

\begin{defn}\label{def:K-contact}
A \emph{$K$-contact structure} $(\eta,\xi,\Phi)$ consists of a contact form $\eta$ on $M$  with Reeb field $\xi$ together with an endomorphism $\Phi:TM\rightarrow TM$ satisfying
\begin{equation*}
\quad\Phi^2=-\id_{TM}+\xi\otimes\eta,\qquad \frakL_\xi \Phi = 0,
\end{equation*}
We require also the following compatibility conditions with $\eta$:
\begin{equation*}
d\eta(\Phi X,\Phi Y)=d\eta(X,Y),\quad d\eta(Z,\Phi Z)>0\quad
\forall X,Y \in TM, Z\in \ker(\eta) \setminus \{0\}
\end{equation*}
\end{defn}
\begin{defn}
Fixing $\eta$, the set of $K$-contact structures $\Phi$ on $(M,\eta)$ is denoted $\scrK_\eta$.
\end{defn}

From Definition~\ref{def:K-contact} one may deduce
$\Phi(\xi)=0$, $\eta\circ\Phi=0$. Moreover, to any $K$-contact structure there belongs a metric $g=g_\Phi$ given by
\begin{equation}\label{metric-eta-Phi}
g_\Phi(X,Y)=d\eta(X,\Phi Y)+\eta(X)\eta(Y).
\end{equation}

The leaves of $\scrF_\xi$ are geodesics with respect to $g$ and the foliation is Riemannian (see \cite[Section 2]{bo-ga}). In particular, we have a \emph{transverse Levi-Civita connection} $D^\T$ on the \emph{normal bundle} $\nu=TM/\mathbb{R}\xi$. This is the unique metric, torsion-free connection on $\nu$ (\emph{i.e.}~$D^\T_X (\pi Y) - D^\T_Y (\pi X) = \pi[X,Y]$ for the projection $\pi\colon TM\rightarrow \nu$).

\begin{rem}
A {\it{Sasakian structure}} is a $K$-contact structure $(\eta,\xi,\Phi,g)$ satisfying the integrability condition $D^g_X\Phi=\xi\otimes X^{\flat_g}-X\otimes\eta\enskip (\forall X\in \frakX(M))$ for the Levi-Civita connection $D^g$ on $(M,g)$ and where $X^{\flat_g}=g(\cdot, X)$. It is well-known that this is equivalent to the almost-K\"ahler cone $\left(\mathbb{R}_{>0}\times M,dr^2+r^2g, d\left(\frac{r^2}{2}\eta\right)\right)$ being K\"ahler.
\end{rem}

\subsection{Basic Forms and Transverse Structure}\label{transverse_section}

Let $\scrF$ be a foliation on $M$ given by an integrable subbundle $T\scrF$ of $TM$. A $p$-form $\alpha$ on $M$ is called {\it{basic}} if
\[
	\iota_\xi \alpha = 0,\qquad
	\frakL_\xi \alpha = 0\qquad \forall \xi \in \Gamma(M,T\scrF).
\]
Let $\left(\Omega^*_B(M), d_B\right)$ denote the subcomplex of basic forms of the de Rham complex and let $C^\infty_B(M) = \Omega^0_B(M)$.
The \emph{basic cohomology} is $H_\scrF^*(M) = H^*(\Omega_B(M), d_B)$.

A {\it{transverse}} symplectic, almost-complex, or Riemannian structure is a corresponding structure on the normal bundle $\nu$ whose Lie derivative in direction of vectors tangent to the leaves  vanishes (see~\cite{tondeur}).
For example, a $K$-contact structure $(\eta,\xi,\Phi, g)$ gives a transverse almost complex structure $\Phi^\T=\Phi|_\nu\in \End(\nu)$, metric $g^\T = g_\Phi |_\nu$, and symplectic form $\omega^\T=d\eta|_\nu \in \Gamma(M,\Lambda^2 \nu^*)$ for $\scrF=\scrF_\xi$.
\begin{defn}
Fixing $\xi$ and $J\in \End(\nu)$, let $\scrK(\xi, J)$ be the space of all $K$-contact structures $(\eta, \xi, \Phi)$ with Reeb field $\xi$ and induced transverse structure $\Phi^\T = J$.
\end{defn}






We briefly describe a $K$-contact structure in local coordinates (see \cite{he} and also \cite{bo-ga,fu-on-wa} in the Sasakian case). We may pick contact Darboux coordinates \cite[Theorem~2.5.1]{geiges} which means the contact form may be written as
\[
	\eta = dx^0 + \sum_{i=1}^n x^{2i-1} dx^{2i},
	\qquad \xi = \partial/\partial x^0.
\]
Then $\omega^T = \sum_{i=1}^n dx^{2i-1}\wedge dx^{2i}$.
The subspace $\ker \eta$, which identifies via $X\mapsto X-\eta(X)X$ with the normal bundle $TM/\mathbb{R}\xi = \nu$, is spanned by
\[
e_{2i-1} = \partial / \partial x^{2i-1},\quad
e_{2i} = \partial / \partial x^{2i} - x^{2i-1} \partial/\partial x^0\quad
(1\leq i \leq n)
\]
Using $\ker(\eta) \bot \xi$, the metric \eqref{metric-eta-Phi} has $g_{ij} = g(e_i,e_j) = g(\partial/\partial x^i, \partial/\partial x^j)$ for $i,j\geq 1$  and $g_{00} = 1$. $\Phi$ is described by the basic functions $\Phi(e_i) = \Phi^j_i e_j$. We have
\begin{equation}\label{coord-expressions}
	\Phi_i^k \Phi_k^j = -\delta_i^j,\qquad
	g_{jk}\Phi_i^j = g_{ji}\Phi_k^j.
\end{equation}

\subsection{Hermitian curvature}
The {\it{\textup(transverse\textup) Hermitian connection}} $\bar\nabla^\T$ on $\nu$ may be defined using the transverse Levi-Civita connection $D^\T$ via
\[
	\bar\nabla^\T_XY=D^\T_XY-\frac{1}{2}\Phi^\T\left(D^\T_X\Phi^\T\right)Y.
\]
 (see \cite{gau-1,lib} and \cite[Sections 9.2, 9.3]{gau} for details on Hermitian connections.) This gives the unique connection on $\nu$ with $\bar\nabla^\T h = 0$, where $2h=g^\T-i\omega^\T$, and whose torsion is the transverse Nijenhuis tensor \eqref{transverse-niej}. Let $\bar{R}^\T$ be the curvature of $\bar\nabla^\T$ and
 \[
	\bar\rho^\T(X,Y)=-\Lambda_\omega(\Phi^\T\circ \bar{R}^\T_{X,Y}),
 \]
where $\Lambda_\omega$ denotes the adjoint of $\omega^\T\wedge -$ on basic forms, see \eqref{trace-operator}. The {\it{\textup(transverse\textup) Hermitian scalar curvature}} is defined via the \emph{Hermitian Ricci $2$-form} $\bar\rho^\T$ as
\[
\bar{s}^\T=2\Lambda_\omega(\bar\rho^\T).
\]
If $\eta$ is fixed we shall emphasize the dependence on $\Phi$ by writing $\bar{s}^\T_\Phi$.


%

%

%


\section{Basic cohomology of $K$-contact structures}\label{sec3}

Throughout this section, fix a compact $K$-contact manifold $(M,\eta,\xi,\Phi,g)$ of dimension $2n+1$ with transverse almost complex structure $J=\Phi^\mathrm{T}$. Our first goal is to describe in Theorem~\ref{space_of_deformation} the space $\mathcal{K}(\xi,J)$ when $2n+1=5$. This requires the development of some machinery of `almost K\"ahler geometry in the transverse,' for example the {\it{transverse $\partial\overline\partial$-Lemma}} (generalizing a result of El-Kacimi-Alaoui~\cite{kac} to the $K$-contact case).

\subsection{Transverse Almost K\"ahler Geometry}

The endomorphism $\Phi$ induces an action on basic $p$-forms via
\[
	(\Phi\alpha)(X_1,\cdots,X_p)=(-1)^p\alpha(\Phi X_1,\cdots,\Phi X_p).
\]
For instance $\Phi\eta=0$. This action preserves basic forms $\alpha \in \Omega^p_B(M)$ since
\begin{eqnarray*}
\iota_\xi(\Phi\alpha)(X_2,\cdots,X_p)&=&(-1)^p\alpha(\Phi\xi,\Phi X_2,\cdots,\Phi X_p)=0,\\ 
\mathfrak{L}_\xi(\Phi\alpha)&=&(\mathfrak{L}_\xi\Phi)\alpha=0.
\end{eqnarray*}
On $p$-forms we have $\Phi^2|_{\Omega^p_B} = (-1)^p \id$. This action coincides for all $\Phi \in \scrK(\xi,J)$ and accordingly we may speak just of $J$-invariant basic forms.

The {\it{twisted exterior derivative}} on $p$-forms is $d^\c=(-1)^p\Phi \,d\,\Phi$ and  preserves basic forms. We write $d^\c_B=d^\c|_{\Omega_B^p}$.
\begin{rem}\label{commute_ddc}
For a basic function $f$ we have $$[(d^\c_Bd_B+d_Bd^\c_B)f] (X,Y)=d^\c_Bf\left(N_{\Phi}(X,Y)\right),$$
using the \emph{transverse Nijenhuis tensor}
\begin{equation}\label{transverse-niej}
N_\Phi(X,Y)=[\Phi X,\Phi Y]+\Phi^2[X,Y]-\Phi[\Phi X,Y]-\Phi[X,\Phi Y].
\end{equation}
Furthermore, when the $K$-contact structure $(\eta,\xi,\Phi,g)$ is Sasakian, $N_\Phi=-2\xi\otimes d\eta$ (see for instance~\cite[p. 204]{bo-ga}) and thus $d^\c_Bd_B+d_Bd^\c_B=0$.
\end{rem}

{\it{The tranverse Hodge star operator}} $\bar{\ast}$ (see~\cite[p. 215]{bo-ga} or~\cite{tondeur}) is defined in terms of the usual Riemannian Hodge operator $\ast=\ast_g$ by setting
\[
\bar{\ast}\alpha=\ast(\eta\wedge\alpha)=(-1)^p\iota_\xi(\ast\alpha),\qquad \alpha\in\Omega_B^p.
\]
In particular, it maps basic forms to themselves and on $p$-forms we have $$(\bar{\ast})^2|_{\Omega^p_B}=(-1)^p.$$

As in almost K\"ahler geometry, define also \emph{adjoint differentials} by $\delta_B=-\mathbin{\bar{\ast}}d_B \mathbin{\bar{\ast}}$ and $\delta^\c_B=-\mathbin{\bar{\ast}}d^\c_B\mathbin{\bar{\ast}}$. The {\it{basic \textup(twisted\textup) Laplacian}} is then given by 
\begin{equation}\label{def:laplacians}
\Delta_B=d_B\delta_B+\delta_Bd_B,\qquad \Delta_B^\c=d_B^\c\delta_B^\c+\delta_B^\c d_B^\c.
\end{equation}

One can easily check the following (similarly to the almost-K\"ahler case, see~\cite{gau}):
\begin{lem}\label{adjoint_expr}
Let $\nabla$ be a torsion free connection on $M$ with $\nabla (d\eta)=0$. For a local, positively oriented, $g$-orthonormal basis $\xi,e_1,\ldots, e_{2n}$ of $TM$ we have
\[\delta^\c_B\alpha=-\sum_{i=1}^{2n}\iota_{\Phi e_i}\left(\nabla_{e_i}\alpha\right),\qquad \alpha \in \Omega^p_B.\]
\textup(such a connection $\nabla$ can be constructed similarly to the Levi-Civita connection.\textup)
\end{lem}

We introduce also the adjoint operators $L\colon \Omega^p_B \rightarrow \Omega^{p+2}_B$ and $\Lambda\colon \Omega^{p+2}_B \rightarrow \Omega^{p}_B$ by
\begin{equation}\label{trace-operator}
	L(\alpha)=\alpha\wedge d\eta,\qquad \Lambda = -\mathbin{\bar\ast}L\mathbin{\bar\ast}.
\end{equation}

\begin{prop}\label{lem:kaehler-identities}
We have the K\"ahler identities on basic forms:
\begin{equation}\label{basic-identities}
[L,\delta^\c_B]=-d_B,\quad [L,\delta_B]=d^\c_B,\quad [\Lambda,d^\c_B]=\delta_B,\quad [\Lambda,d_B]=-\delta^\c_B
\end{equation}
\end{prop}
\begin{proof}
Using $\Phi^2\alpha=(\bar{\ast})^2\alpha=(-1)^p\alpha$, we reduce to proving only the first identity. Choose an oriented, local, $g$-orthonormal basis $\{\xi,e_i\}$. For a basic $p$-form $\alpha$ we then compute, using Lemma~\ref{adjoint_expr} and a torsion free connection $\nabla$ preserving $d\eta$:
\begin{align*}
[\delta^\c_B,L]\alpha&=\delta^\c_B(\alpha\wedge d\eta)-d\eta\wedge\delta^\c_B\alpha\\
&=\sum_{i=1}^{2n} -\iota_{\Phi e_i}\left(\nabla_{e_i}\left(\alpha\wedge d\eta\right)\right)+d\eta\wedge\iota_{\Phi e_i}(\nabla_{e_i}\alpha)\\
&=\sum_{i=1}^{2n} -\iota_{\Phi e_i}\left(\nabla_{e_i}\alpha \wedge d\eta\right)+d\eta\wedge\iota_{\Phi e_i}(\nabla_{e_i}\alpha)\\
&=\sum_{i=1}^{2n} -\iota_{\Phi e_i}\left(d\eta\right)\wedge\nabla_{e_i}\alpha\\
&=\sum_{i=1}^{2n} e_i^{\flat_g}\wedge\nabla_{e_i}\alpha=d_B\alpha.
\end{align*}
Note here that the expression for $d\alpha$ in terms of covariant derivatives reduces on basic forms to the last equation since $\nabla_\xi \alpha = 0$ for basic $\alpha$. To see this, note that $\nabla(d\eta)=0$ implies $d\eta(\nabla_X \xi, Z) = 0$ for any $X,Z$ and the Reeb field $\xi$. Therefore $\nabla_X \xi$ is a multiple of $\xi$, so for basic forms $i_{\nabla_X \xi} \alpha = 0$. Using that $\nabla$ is torsion free,
\[
	0=\frakL_\xi \alpha(X_1,\ldots, X_p) = \nabla_\xi \alpha(X_1,\ldots, X_p) + \sum_{k=1}^p (-1)^k\alpha(\nabla_{X_k} \xi, X_1, \ldots \widehat{X_k} \ldots, X_p),
\]
for arbitrary tangent vectors $X_k$, and $\nabla_\xi \alpha=0$ follows.
\end{proof}
\begin{rem}\label{equal_laplacian}
It follows from Remark~\ref{commute_ddc} and the K\"ahler identities~\eqref{basic-identities} that $\Delta_B=\Delta_B^\c$ whenever the $K$-contact structure $(\eta,\xi,\Phi,g)$ is Sasakian.
\end{rem}

The Laplacians \eqref{def:laplacians} are basic transversely elliptic operators (see~\cite{kac}). Hence they are Fredholm operators, so we get \emph{basic Green operators} $\G_B$ and $\G_B^\c$ with
\begin{equation}\label{green-operators}
	\G_B\Delta_B \alpha=\Delta_B \G_B \alpha = \alpha + (\alpha)_\mathrm{H},\quad
	\G_B^\c\Delta_B^\c \alpha=\Delta_B^\c \G_B^\c \alpha = \alpha + (\alpha)_\mathrm{H^\c},
\end{equation}
where $(\alpha)_\mathrm{H}$, $(\alpha)_{\mathrm{H}^\c}$ are orthogonal projections onto $\Delta_B$ and $\Delta_B^\c$ harmonic forms. Recall that $\G_B, \Delta_B$ commute with $d_B, \delta_B$, while $\G_B^\c, \Delta_B^\c$ commute with $d_B^\c, \delta_B^\c$.

\begin{lem}\label{general-ddc-lem}
For any $d_B$-exact, $d_B^\c$-closed $\alpha\in \Omega^p_B$ there exists $\psi \in \Omega^{p-2}_B$ with
\begin{equation}
\alpha=\mathbb{G}_Bd_Bd^\c_B\psi=d_B\mathbb{G}_Bd^\c_B\psi.
\end{equation}
\end{lem}
\begin{proof}
This is an analogue of~\cite[Lemma 3.1]{lej}. Decomposing $\alpha$ with respect to $\Delta_B$ using \eqref{green-operators} gives $\alpha=d_B\delta_B\mathbb{G}_B\alpha$. With respect to $\Delta_B^\c$ we have
$\alpha=\left(\alpha\right)_{\mathrm{H}^\c}+d^\c_B\delta^\c_B\mathbb{G}_B^\c\alpha$.
Now we apply the K\"ahler identities \eqref{basic-identities} to deduce
\begin{align*}
\alpha&=d_B\delta_B\mathbb{G}_B\alpha=d_B\delta_B\mathbb{G}_B\left[\left(\alpha\right)_{\mathrm{H}^\c}+d^\c_B\delta^\c_B\mathbb{G}_B^\c\alpha\right]\\
&=d_B\mathbb{G}_B\delta_B\left(\alpha\right)_{\mathrm{H}^\c}+d_B\mathbb{G}_B\delta_Bd^\c_B\delta^\c_B\mathbb{G}_B^\c\alpha\\
&=d_B\mathbb{G}_B [\Lambda,d^\c_B]\left(\alpha\right)_{\mathrm{H}^\c}-d_B\mathbb{G}_Bd^\c_B\delta_B\delta^\c_B\mathbb{G}_B^\c\alpha\\
&=-d_B\mathbb{G}_Bd^\c_B \Lambda\left(\alpha\right)_{\mathrm{H}^\c}-d_B\mathbb{G}_Bd^\c_B\delta_B\delta^\c_B\mathbb{G}_B^\c\alpha\\
&=d_B\mathbb{G}_Bd^\c_B \left(-\Lambda\left(\alpha\right)_{\mathrm{H}^\c}-\delta_B\delta^\c_B\mathbb{G}_B^\c\alpha\right).\qedhere
\end{align*}
\end{proof}


\begin{cor}\label{ddc-lemma} 
Suppose $\omega_1,\omega_2\in\Omega_B^2$ are $d_B$-closed, $J$-invariant, and basic cohomologous. Then there exists a smooth basic function $f$ with $\omega_2=\omega_1+d_B\mathbb{G}_Bd^\c_B f$.
\end{cor}
\begin{rem}
Corollary~\ref{ddc-lemma} generalizes the {{transverse $\partial\overline\partial$-Lemma}}~\cite{kac}. Indeed, Remark~\ref{equal_laplacian}
implies that in the Sasakian case $\omega_2=\omega_1+d_B\mathbb{G}_Bd^\c_B f=\omega_1+d_Bd^\c_B\mathbb{G}_B f$.
\end{rem}

\subsection{The Space $\scrK(\xi,J)$ in Dimension $5$}

In dimension $2n+1$ the bundle $\Lambda^2_B$ of basic $2$-forms decomposes into $\pm1$-eigenspaces of the transverse Hodge operator
\[
	\Lambda^2_B=\Lambda^+_B\oplus\Lambda^-_B.
\]
Similarly $\Lambda^2_B = \Lambda^{J,-}_B \oplus \Lambda^{J,+}_B$ into the $\pm1$-eigenspaces of $J$. In dimension $5$ we have
\begin{equation}\label{decomp-2-forms}
\Lambda^+_B=\mathbb{R}\,.\,d\eta\oplus \Lambda^{J,-}_B,\qquad\Lambda^-_B= \Lambda^{J,+}_{0},
\end{equation}
where $\Lambda^{J,+}_{0}$ is the subbundle of $J$-invariant $2$-forms pointwise orthogonal to $d\eta$.\smallskip

We denote by $b_B^+$ (resp. $b_B^-$) the dimension of the space of $\Delta_B$-harmonic $\bar{\ast}$-self-dual
(resp.~$\bar{\ast}$-anti-self-dual) basic $2$-forms. Since $\Delta_B$-harmonic basic $2$-forms are preserved by the $\bar{\ast}$-operator, the dimension of the basic cohomology is
\[
	b^2_B := \dim H^2_{\mathcal{F}_\xi}(M)=b_B^+ + b_B^-.
 \]
Let $h_{J,B}^-$ be the dimension of the $\Delta_B$-harmonic $J$-anti-invariant basic $2$-forms. It is easy to see that this definition agrees with that in~\cite{dra-li-zha}.

\begin{prop}\label{J-invariance}
Let $(M,\eta,\xi,\Phi,g)$ be a $5$-dimensional compact $K$-contact manifold. If $h_{J,B}^-=b_B^+-1$, then for any basic function $f$, $d_B\mathbb{G}_Bd^\c_B f$ is $J$-invariant.
\end{prop}
\begin{proof}
The proof is similar to that of~\cite[Proposition 2]{lej}, so we only sketch the argument. Beginning with \eqref{decomp-2-forms}, a computation using the K\"ahler identities \eqref{basic-identities} shows that the $J$-anti-invariant part of $d_B\mathbb{G}_Bd^\c_B f$ is 
\[
(d_B\mathbb{G}_Bd^\c_B f)^{J,-}=\frac{1}{2}\left(f_0d\eta\right)_\mathrm{H}-\frac{1}{4}g\big(\left(f_0d\eta\right)_\mathrm{H},d\eta\big)d\eta,
\]
for the orthogonal projection $f_0$ of $f$ onto the complement of the constants and
the $\Delta_B$ harmonic part $\left(f_0d\eta\right)_\mathrm{H}$. The condition $h_{J,B}^-=b_B^+-1$ implies $\left(f_0d\eta\right)_\mathrm{H}=0$. 
\end{proof}

Using Remark~\ref{equal_laplacian} we see $h_{J,B}^-=b_B^+-1$ when the $K$-contact structure is Sasakian. It is well-known that an almost complex structure is integrable precisely when $(dd^\c + d^\c d)f$ is $J$-invariant for every function $f$. The condition $h_{J,B}^-=b_B^+-1$ appears therefore as a semi integrability condition.

\begin{defn}\label{def:semiSasakian}
A $K$-contact structure is \emph{semi-Sasakian} if $h_{J,B}^-=b_B^+-1$ holds.
\end{defn}

\begin{rem}\label{upper-continuous}
Under a smooth variation of the transverse almost-complex structure $J_t$, the dimension $h_{J_t,B}^-$ is an upper semi-continuous function of $t$. To see this, consider the family of basic transversally strongly elliptic differential operators
\[
P_t \colon \Omega^{J_t,-}_B\rightarrow\Omega^{J_t,-}_B,\quad
\alpha \mapsto \left(d\delta^{t}_B\alpha\right)^{J_t,-}.
\]
Here, $\Omega_B^{J_t,-}$ are basic $J_t$-anti-invariant $2$-forms, $(\,)^{J_t,-}$ is the projection, and $\delta^{t}_B$ is the adjoint of $d_B$ with respect to the $K$-contact metric induced by $J_t$. Then $h_{J_t, B}^-$ is the kernel of $P_t$, whose dimension is an upper semi-continuous function (using \cite[Theorem~6.1]{kac-gmi}, an adaption of \cite[Theorem~4.3]{kod-mor}, see also~\cite{dra-li-zha}).
\end{rem}

We may now generalize \cite[Proposition 7.5.7]{bo-ga} to the semi-Sasakian case:

\begin{thm}\label{space_of_deformation}
Let $(M,\eta,\xi,\Phi,g)$ be a $5$-dimensional compact semi-Sasakian manifold with transverse structure $J=\Phi^\T$. Then we have a diffeomorphism
\begin{equation}\label{K-identification}
\mathcal{K}(\xi,J) \simeq\mathcal{H}\times C^\infty_{B,0}(M)\times H^1(M,\mathbb{R}),
\end{equation}
for the basic functions with zero integral $C^\infty_{B,0}(M)$ and
where
\[
\mathcal{H}  =\left\{f\in C^\infty_{B,0}(M) \enskip\middle|\enskip (d\eta+d_B\mathbb{G}_Bd^\c_B f)(X,\Phi X)>0\enskip\forall X\in\ker\left(\eta+\mathbb{G}_Bd^\c_B f\right)\right\}.
\]
\end{thm}
\begin{proof}
Now that Corollary~\ref{ddc-lemma} and Proposition~\ref{J-invariance} are in place,
the proof runs parallel to the Sasakian case (see~\cite[Proposition 7.5.7]{bo-ga}).
We pick a basis $[\alpha_1]_B, \ldots, [\alpha_N]_B$ of the space $H^1_{\mathcal{F}_\xi}(M)\cong H^1(M,\mathbb{R})$ \cite[Proposition 7.2.3]{bo-ga}. Note the diffeomorphism
\[
	\psi\colon H^1_{\scrF_\xi}(M) \times C^\infty_{B,0}(M) \rightarrow \Omega^1_{B,\text{closed}}(M),\quad
	\left(\sum\lambda_i[\alpha_i]_B, g \right)\mapsto \sum \lambda_i \alpha_i + dg,
\]
endowing $H^1_{\scrF_\xi}(M)$ with the norm-topology of finite-dimensional vector space.

Suppose $(\overline\eta,\xi,\overline\Phi,\overline{g})\in\mathcal{K}(\xi,J)$. Then $\overline\eta-\eta$ is a basic $1$-form, so $[d\overline\eta]_B=[d\eta]_B$. By Corollary~\ref{ddc-lemma}, there exists a unique $f\in C^\infty_{B,0}$
with $d\overline\eta=d\eta+d_B\mathbb{G}_Bd^\c_B f$, depending smoothly on $\overline\eta$, \emph{e.g.}~because of the formula for $f$ in the proof of Lemma~\ref{general-ddc-lem}.
The $1$-form $\alpha=\overline\eta-\eta-\mathbb{G}_Bd^\c_B f$ is closed and basic. Letting $\psi^{-1}(\alpha) = ([\alpha]_B, g)$, we may define the smooth one-to-one map \eqref{K-identification} by $(\overline\eta,\xi,\overline\Phi,\overline{g}) \mapsto (f, g, [\alpha]_B)$.

Conversely, given $(f,g,[\alpha]_B)\in\mathcal{H}\times C^\infty_{B,0}(M)\times H^1(M,\mathbb{R})$, let $\psi(g,[\alpha]_B) = \beta$.
Define $\overline\eta=\eta+\mathbb{G}_Bd^\c_B f+\beta$. Since $i_\xi \overline (\eta\wedge (d\overline\eta)^2)=(d\eta+\G_B d_B d_B^\c f)^2$, the positivity assumption on $f \in \mathcal{H}$ implies that $\overline\eta$ is a contact form. Moreover, using Proposition~\ref{J-invariance} one easily checks that
\[
\overline\Phi=\Phi-\xi\otimes\left(\overline\eta-\eta\right)\circ\Phi,\qquad
\overline{g}=d\overline\eta\circ(\overline\Phi\otimes \id)+\overline\eta\otimes\overline\eta,
\]
determine a $K$-contact structure $(\overline\eta,\xi,\overline\Phi,\overline{g})$.
\end{proof}

\section{Extremal $K$-contact metrics}\label{sec4}
Throughout this section, $(M,\eta)$ is a compact contact manifold of dimension $2n+1$ with volume form $dv_\eta={(2n!)}^{-1}\eta\wedge\left(d\eta\right)^{2n}$.
The aim of this section is to introduce natural representatives of $K$-contact structures on $(M,\eta)$ via a moment map set-up. Once we have proven the main Theorem~\ref{moment_map}, we may use this description of extremal metrics to establish easily a number of interesting consequences.


\subsection{The Transverse Hermitian Scalar Curvature as a Moment Map}

Let $\mathcal{K}_\eta$ be the Fr\'echet space of $K$-contact structures compatible with $\eta$, equipped with its formal K\"ahler structure $(\bf{\Omega},\bf{J})$ described in~\cite{he} (which is analogous to~\cite{don} in the K\"ahler and \cite{fuj} in the almost-K\"ahler case). We have
\[
T_\Phi \mathcal{K}_\eta=\left\{A\in \mathrm{End}(TM) \middle| A\xi=0,\mathfrak{L}_\xi A=0,A\Phi+\Phi A=0, d\eta(A\cdot,\cdot)+d\eta(\cdot, A\cdot)=0\right\}. 
\]

The K\"ahler form is defined at the point $\Phi \in \scrK_\eta$ by
\[
{\bf{\Omega}}_\Phi(A,B)=\int_M \mathrm{trace}(\Phi\circ A\circ B)\, dv_\eta,
\]
while an ${\bf{\Omega}}$-compatible almost-complex structure (in fact integrable) is given by
\[
	{\bf{J}}_\Phi A=\Phi\circ A.
\]

The action of the strict contactomorphism group $\gamma \in \Con(M,\eta)$ on $\mathcal{K}_\eta$ via ${\gamma_\ast \circ \Phi \circ \gamma_\ast^{-1}}$  is by symplectomorphisms, observing that $\gamma^* dv_\eta = dv_\eta$.

\begin{thm}\label{moment_map}\textup(see~\cite[Remark 4.3]{he}\textup)
The action of $\Con(M,\eta)$ on $\scrK_\eta$ is Hamiltonian. The moment map $\mu\colon \mathcal{K}_\eta\longrightarrow \Lie\Con(M,\eta) ^\ast \cong C^\infty_B(M)^*$ is
\[
\mu(\Phi)(f)=-\int_M \bar{s}_\Phi^\T f\,dv_\eta,
\]
for the transverse Hermitian scalar curvature $\bar{s}^\T_\Phi$ of $\Phi$ and where we use \eqref{contact-hamiltonian}.
\end{thm}
\begin{proof}
The infinitesimal action of $X \in \Lie \Con(M,\eta)$ at a point $\Phi\in\mathcal{K}_\eta$ is given by $\widehat{X}(\Phi) = -\mathfrak{L}_X \Phi$. For a tangent vector $A\in T_\Phi \scrK_\eta$ let $Q(A) \in C^\infty_B$ be the derivative of the map $\Phi \mapsto \bar{s}_\Phi^\T$ in direction $A$.
We must show
\begin{equation}\label{toshow}
	\mathbf{\Omega}_\Phi( -\frakL_{X_f}\Phi, A) = -\int_M f\cdot Q(A) dv_\eta.
\end{equation}



We do local computations as in~\cite{he}. Pick local coordinates as in Subsection~\ref{transverse_section}. We shall write $f_k=f_{,k} = \partial f/\partial x^k$ and $f_{;k}$ for covariant differentiation.

From $d\eta(X_{f}\,,\cdot)=-df$ we have (we adopt the summation convention that roman indices, if appearing twice, range over $1,\ldots, n$, \emph{i.e.}~excluding zero)
\begin{equation}\label{Xf}
X_f=-\frac{1}{2}\Phi^j_kg^{ki}f_j\frac{\partial}{\partial x^i} + f\frac{\partial}{\partial x^0}
\end{equation}
Since $\Phi$ is basic, we may locally write
\[
\mathcal{L}_{X_f}\Phi=B^i_j\,\frac{\partial}{\partial x^i}\otimes dx^j.
\]
The local coordinate formula for the Lie derivative combined with \eqref{Xf} gives
\[
B^i_j=
-\frac{1}{2}\Phi^i_{j,p}\Phi^l_kg^{kp}f_l
+\frac{1}{2}\Phi^p_j\left(\Phi^l_kg^{ki}f_l\right)_{,p}
-\frac{1}{2}\Phi^i_p\left(\Phi^l_kg^{kp}f_l\right)_{,j}.
\]
From \eqref{coord-expressions} we see then
\[
\Phi^s_iB^i_j=-\frac{1}{2}\Phi^s_i\Phi^i_{j,p}\Phi^l_kg^{kp}f_l+\frac{1}{2}\Phi^s_i\Phi^p_j\left(\Phi^l_kg^{ki}f_l\right)_{,p}
+\frac{1}{2}\left(\Phi^l_kg^{ks}f_l\right)_{,j}.
\]
Therefore we have
\begin{align*}
& \mathbf{\Omega}_\Phi( \frakL_{X_f}\Phi, A)= \int_M \mathrm{trace}(\Phi\circ \mathcal{L}_{X_f}\Phi\circ A)\,dv_\eta\\
=&\int_M\left(-\frac{1}{2}\Phi^s_i\Phi^i_{j,p}\Phi^l_kg^{kp}f_lA^j_s+\frac{1}{2}\Phi^s_i\Phi^p_j\left(\Phi^l_kg^{ki}f_l\right)_{,p}A^j_s
+\frac{1}{2}\left(\Phi^l_kg^{ks}f_l\right)_{,j}A^j_s\,\right)dv_\eta.
\end{align*}
Let $C^s_j = \Phi_i^s \Phi^i_{j,p} \Phi^l_k g^{kp} f_l$. Using \eqref{coord-expressions} and its derivative one checks ${g_{sl} C^s_j = - g_{sj} C^s_l}$, so $C$ is $g^\T$-anti-symmetric. On the other hand, $A^j_s$ is $g^\T$-symmetric and so the trace $C_j^k A^j_s$, the first summand in the bracket, vanishes.

The second and third summand are equal (from $\Phi^k_j A^j_s = - \Phi^j_s A^k_j$) so
\[
\mathbf{\Omega}_\Phi( \frakL_{X_f}\Phi, A)=\int_M\left(\Phi^l_kg^{ks}f_l\right)_{,j}A^j_s\,dv_\eta.
\]
The variation of the transverse Hermitian scalar curvature (see~\cite{gau}) is given in terms of the variation of $\Phi$ along $A$ by
\[
{\dot{\bar{s}}}^\T=Q(A)= -(g^{ks}\Phi^l_k (A^j_s)_{;j})_{;l}.
\]
Using integration by parts twice (justified as in~\cite{he}),
we conclude the proof of \eqref{toshow}:
\[
	\int_M\left(\Phi^l_kg^{ks}f_l\right)_{,j}A^j_s\,dv_\eta
	= -\int_M \Phi^l_kg^{ks}f_l (A^j_s)_{;j} dv_\eta
	= \int_M f \left( \Phi^l_kg^{ks} (A^j_s)_{;j} \right)_{;l}dv_\eta\qedhere
\]
\end{proof}

Much of the above works in greater generality; instead of a contact manifold, begin with a closed manifold $M$ and closed $2$-form $\omega$ with $\omega^q$ never zero and $\omega^{q+1}=0$. This amounts to a codimension $2q$ foliation $T\scrF=\ker \omega$ with transverse symplectic structure. The argument of Proposition~\ref{lem:kaehler-identities} gives K\"ahler identities on basic forms. One may then define a Fr\'echet space $\mathcal{AC}(\omega)$ of $\omega$-compatible transverse almost complex structures $J$, which has a K\"ahler structure. Combining $J$ with $\omega$ yields a bundle-like metric $g$, so again we obtain a transverse Levi-Civita connection $D^\T$ and Hermitian connection and corresponding scalar curvatures.

Introducing the variation of connection $\dot\alpha$ in direction $A\in T_J \mathcal{AC}(\omega)$ in the standard way (see~\cite[Section~9.5]{gau}) defines, using that the connection is basic (see~\cite[Proposition~3.6]{tondeur}), a basic $1$-form. The Mohsen formula $2\dot\alpha = g(\delta A, \cdot)$ for $\delta A=-\sum i_{e_k} D^\T_{e_k} A$, summing over an orthonormal frame $e_k$ of the normal bundle, can be established using the argument of~\cite[Proposition~9.5.1]{gau}. The variation $\dot\alpha$ being basic, the K\"ahler identities then give $\bar{s}_J^\T = -\delta J (\delta A)^\flat$.

It is not hard to define a transverse Hamiltonian group that acts on $\mathcal{AC}(\omega)$. As in Theorem~\ref{moment_map}, this action is Hamiltonian with moment map $J\mapsto \bar{s}^\T_J$.

\begin{defn}
The square-norm of the moment map defines a functional
\begin{equation}\label{square_norm}
\mathfrak{C}\colon \scrK_\eta \rightarrow \mathbb{R},\quad
\mathfrak{C}(\Phi) = \|\mu(\Phi)\|^2=\int_M ( \bar{s}^\T_\Phi)^2\,dv_\eta.
\end{equation}
The critical points of this functional are called {\it{extremal $K$-contact metrics}}.
\end{defn}

Given a $K$-contact structure $(\eta,\xi,\Phi,g)$, we denote by $X_{\bar{s}^\T} \in \Lie\Con(M,\eta)$ the vector field belonging via \eqref{contact-hamiltonian} to the scalar curvature $\bar{s}_\Phi^\T$.\medskip

\begin{prop}~\label{prop-critical-points}
$\Phi$ is extremal if and only if $X_{\bar{s}^\T}$
is a Killing vector field with respect to the metric $g_\Phi$ induced by $\Phi$ \textup(equivalently, when $\frakL_{X_{\bar{s}^\T}} \Phi = 0$\textup).
\end{prop}
\begin{proof}
This follows from the moment map set-up (see~\cite{apo-dra,lej-1} in the case of extremal almost-K\"ahler metrics).
For $A\in T_\Phi \scrK_\eta$ the differential of \eqref{square_norm} in direction $A$ is
\[
	\frak{C}_{*,\Phi}(A) = 2\langle \mu_{*,\Phi}(A), \mu(\Phi)\rangle
	=2d \mu^{\bar{s}^\T_\Phi}(A) = 2\mathbf{\Omega}_\Phi(-\frakL_{X_{\bar{s}^\T_\Phi}}\Phi, A),
\]
where we write $\mu^f = \mu(\cdot)(f)$. The last equality is by Theorem~\ref{moment_map}.
\end{proof}
\begin{ex}
Calabi's extremal problem~\cite{cal-1,cal-2} was extended to Sasaki geometry by Boyer--Galicki--Simanca in~\cite{bo-gal-sim,bo-gal-sim-1} where they introduce the notion of {\it{extremal Sasakian metrics}}. This notion generalizes {\it{Sasaki-Einstein metrics}}
(more generally the so-called $\eta$-Einstein metrics, see for instance~\cite{spa})
and constant scalar curvature Sasaki metrics. 
From Proposition~\ref{prop-critical-points}, extremal Sasakian metrics $\Phi$ are extremal $K$-contact metrics. Indeed, when $\Phi$ is Sasakian the  Riemannian scalar curvature coincides with the Hermitian scalar curvature $\bar{s}^\T$. 
\end{ex}

\begin{rem}
Extremal $K$-contact metrics are a natural extension of extremal Sasakian metrics~\cite{bo-gal-sim,bo-gal-sim-1} to $K$-contact geometry.
Given a background Sasakian structure $(\eta,\xi,\Phi,g),$ Boyer--Galicki--Simanca consider the space $\mathcal{S}(\xi,J)$, of Sasakian structures with
common Reeb vector field $\xi$ and transverse {\it{integrable}} almost-complex structure $J=\Phi^\T$, arising from deforming the contact form $\eta$ by
$\eta\mapsto\eta_t=\eta+t\alpha$, where $\alpha$ is a basic $1$-form with respect to the characteristic foliation $\mathcal{F}_\xi$.
Hence, by {\it {Gray's Stability Theorem}} (see~\cite{gra} or for instance~\cite[p. 190]{bo-ga}), there exist
a diffeomorphism $\gamma$ such that $\gamma^\ast\eta'=\eta,$ for any $(\eta',\xi,\Phi',g')\in\mathcal{S}(\xi,J)$.
This gives a new Sasakian structure $(\eta,\xi,\gamma_\ast^{-1}\Phi'\gamma_\ast,\gamma^\ast g')$ with $\gamma_\ast^{-1}\Phi'\gamma_\ast\in\mathcal{K}_\eta$.
\end{rem}

\section{A $K$-contact Futaki invariant}\label{sec5}

We continue to draw consequences of Theorem~\ref{moment_map}.
In this section, we generalize the Futaki invariant from~\cite{bo-gal-sim}, \cite{fut} to the non-integrable $K$-contact setting. Fix throughout a $(2n+1)$-dimensional compact contact manifold $(M,\eta)$ with volume form $dv_\eta={(2n)!}^{-1}\eta\wedge\left(d\eta\right)^{2n}$. Moreover, we shall assume $\scrK_\eta \neq \emptyset$. In this case we may find by \cite[Corollary~4.14]{boyer-torus} a maximal torus $G \subset \Con(M,\eta)$ in the \emph{strict} contactomorphism group of \emph{Reeb type}, meaning $\xi \in \Lie(G)$. Note that in $\Con(M,\eta)$ different maximal tori are not necessarily conjugate.

Since $G$ is compact, an averaging argument shows that the subspace $\scrK_\eta^G \subset \scrK_\eta$ of $G$-invariant $K$-contact structures is contractible as well.\smallskip

Let $\Pi^G$ be the orthogonal projection from $C^\infty_{B}(M)$, the space of basic functions, onto $\mathfrak{g}_\eta$ the contact Hamiltonians of $\Lie(G)$, recalling the identification~\eqref{contact-hamiltonian}.\smallskip

As a generalization of \cite{sim} in the K\"ahler case we find:

\begin{prop}\label{independ_proj}
For every smooth curve $\Phi_t \in\mathcal{K}_\eta^G$ the projection of the Hermitian scalar curvature $\Pi^G (\bar{s}_{\Phi_t}^\T) \in\mathfrak{g}_\eta$ is independent of $t$.
\end{prop}
\begin{proof}
We may equivalently show that $\mu|_{\mathcal{K}_\eta^G}(\Phi_t)(X)$ is constant for any $X \in \Lie(G)$:
\begin{align*}
\left.\frac{d}{dt}\right|_{t_0} \mu(\Phi_t)(X) &= d\mu(\dot{\Phi}(0))(X)={\bf\Omega}_{\Phi_{t_0}}(\dot{\Phi}(0),\widehat{X}\left({\Phi_{t_0}}\right))=0,
\end{align*}
using that the infinitesimal action $\widehat{X}\left({\Phi_{t_0}}\right)$ vanishes ($\Phi_{t_0}$ being $G$-invariant).
\end{proof}

\begin{defn}
For fixed $G\subset \Con(M,\eta),$ we define the vector field $\mathrm{Z}^G_\eta\in \Lie(G)$ corresponding to the contact Hamiltonian $\mathrm{z}^G_\eta=\Pi^G \bar{s}^\T_\Phi \in \mathfrak{g}_\eta$, via \eqref{contact-hamiltonian}, using an arbitrary $K$-contact structure $\Phi \in \scrK^G_\eta$.

By Proposition~\ref{independ_proj} and $\scrK_\eta^G \simeq \{\mathrm{pt}\}$, the \emph{extremal vector field} $\mathrm{Z}^G_\eta$ is well-defined (see~\cite{fut-mab} in the K\"ahler case)
\end{defn}

\begin{prop}\label{extreme_cond}
A $K$-contact structure $\Phi\in\scrK_\eta^G$ is extremal precisely when
\begin{equation}\label{extreme-cond-proj}
	\bar{s}^\T_\Phi = \mathrm{z}^G_\eta.
\end{equation}
\end{prop}
\begin{proof}
Suppose $(\eta,\xi,\Phi,g)$ is extremal. By Proposition~\ref{prop-critical-points}, $X_{\bar{s}^\T}$ is a $G$-invariant Killing field.
$\Phi$ being $G$-invariant, $G$ is a subgroup of the isometry group for $(M,g_\Phi)$. Consider the connected Lie subgroup $H \subset \mathrm{Isom}(M,g)$ belonging to the abelian subalgebra $\Lie(G) + \mathbb{R}\cdot X_{\bar{s}^\T}$. The closure $\overline{H}$ is a torus in ${\mathrm{Isom}(M,g) \cap \Con(M,\eta)}$ (and also in $\Con(M,\eta)$), containing $G$. By maximality, $G \subset H=\overline{H} \subset G$, so $X_{\bar{s}^\T} \in \Lie(G)$ and \eqref{extreme-cond-proj} follows.
Conversely, from \eqref{extreme-cond-proj} we have $\bar{s}^\T \in \Lie(G) \subset \Lie (\mathrm{Isom}(M,g))$ so $X_{\bar{s}^\T}$ is a Killing field and $\Phi$ is extremal by Proposition~\ref{prop-critical-points}.
\end{proof}

Consider the `angle' map $\langle\,\cdot\,, \mathrm{Z}^G_\eta\rangle$ on $\Lie(G)$. If $(M,\eta)$ admits an extremal metric $\Phi$, then by the previous proposition the angle map completely determines its scalar curvature $\bar{s}^\T_\Phi$. A more explicit definition of this map is as follows:

\begin{defn}
The {\it{$K$-contact Futaki invariant}} relative to the group $G$ is the map
\[
\mathfrak{F}_{\mathcal{K}_\eta^G}\colon \Lie(G) \rightarrow \mathbb{R},\qquad X\mapsto \int_M \eta(X) \mathring{\bar{s}}^\T dv_\eta,
\]
where $\mathring{\bar{s}}^\T =\bar{s}_\Phi^\T-\frac{\int_M \bar{s}_\Phi^\T dv_\eta}{\int_Mdv_\eta}$ is the zero integral part (for $\Phi\in\scrK_\eta^G$ arbitrary).
\end{defn}

The previous discussion implies (see~\cite[Proposition 5.2]{bo-gal-sim} in the Sasakian case):

\begin{prop}
If $(M,\eta)$ admits an extremal $K$-contact metric, the following are equivalent:
\begin{enumerate}
\item
Every \textup(some\textup) extremal metric has constant Hermitian scalar curvature.
\item
$\mathfrak{F}_{\mathcal{K}_\eta^G}\equiv 0$.
\end{enumerate}
\end{prop}


\begin{prop}
The vector field $\mathrm{Z}^G_\eta$ is invariant under $G$-invariant {\it{strict contact isotopy}} of $\eta$.
\end{prop}
\begin{proof}
Suppose that we have a smooth $G$-invariant family of contact forms $\eta_t$ with the same Reeb vector field $\xi$ (such that $\eta_0=\eta$). Then,
by {{Gray's Stability Theorem}}, there exists
a smooth family of diffeomorphisms $\gamma_t$ such that $\gamma_0=\id$ and $\gamma^\ast_t\eta_t=\eta$. Then, $\gamma^\ast_t\left(\mathrm{Z}^G_{\eta_t}\right)=\mathrm{Z}^G_\eta$.
Moreover, using the $G$-invariance of the vector field generating $\gamma^\ast_t$, we have $\gamma^\ast_t\left(\mathrm{Z}^G_{\eta_t}\right)=\mathrm{Z}^G_{\eta_t}$.
\end{proof}

On a compact contact manifold $(M,\eta)$ consider the space $\mathcal{K}^G(\xi)$ of all $G$-invariant $K$-contact structures with contact forms strictly isotopic to $\eta$ and common Reeb field $\xi$. One easily deduces
the following: if $\mathcal{K}^G(\xi)$ contains a $K$-contact metric with constant transverse Hermitian scalar curvature then $\mathrm{Z}^G_\eta=0$.
Conversely, if $\mathrm{Z}^G_\eta=0,$ any extremal $K$-contact metric in $\mathcal{K}^G(\xi)$ is of constant transverse Hermitian scalar curvature.

Since the Reeb field $\xi$ lies in $\Lie(G)$ it follows that
\[
{\int_Ms^{\nabla^\T}dv_\eta}=\int_M \mathrm{z}^G_\eta dv_\eta
\]
so that
\[
\mathfrak{F}_{\mathcal{K}_\eta^G}(\mathrm{Z}^G_\eta)=\int_M\eta(\mathrm{Z}^G_\eta)\mathring{s^{\nabla^\T}}dv_\eta=\int_M\mathrm{z}^G_\eta\mathring{s^{\nabla^\T}}dv_\eta=\int_M\left(\mathrm{z}^G_\eta\right)^2 dv_\eta-\frac{\left(\int_Ms^{\nabla^\T}dv_\eta\right)^2}{\int_Mdv_\eta}.
\]
We obtain a lower bound for the functional~(\ref{square_norm}):
\begin{prop}
Let $S_\eta=\int_Ms^{\nabla^\T}dv_\eta$ and $V_\eta=\int_Mdv_\eta$. For all $\Phi \in \scrK^G_\eta$ we have
\[
\int_M (s^{\nabla^\T})^2\,dv_\eta\geqslant\,\mathfrak{F}_{\mathcal{K}_\eta^G}(\mathrm{Z}^G_\eta)+\frac{S_\eta^2}{V_\eta}.
\]
Equality holds if and only if $\Phi\in\mathcal{K}_\eta^G$
induces an extremal metric.
\end{prop}
\begin{proof}
The inequality follows from the above discussion. Moreover, equality
holds if and only if $s^{\nabla^\T}=\mathrm{z}^G_\eta$, \emph{i.e.}~by Proposition~\ref{extreme_cond} when $\Phi$ is extremal. 
\end{proof} 

\section{Deformations of Extremal $K$-contact Metrics in Dimension $5$}\label{deform_section}

In the Sasakian setting, Boyer-Galicki-Simanca developed the notion of {\it{Sasaki cone}}~\cite{bo-gal-sim,bo-gal-sim-1} and
proved in~\cite{bo-gal-sim} that the existence of extremal Sasakian metrics is an open condition in the Sasaki cone, as in the K\"ahler set-up~\cite{leb-sim,sim-1,fuj-sch}.

In this section, we show that a similar result holds in the semi-Sasakian case. Let $\scrK^{G,\mathrm{semi}}_\eta$ be the subspace of $\scrK^G_\eta$ of those $\Phi$ that are semi-Sasakian (see Definition~\ref{def:semiSasakian}).

\begin{thm}\label{openess}
Let $(M,\eta)$ be a $5$-dimensional compact contact manifold and $G$
be a maximal torus in $\Con(M,\eta)$. Let $\Phi_t$ be a smooth curve in $\scrK^{G,\mathrm{semi}}_\eta$ with $\Phi_0$ an extremal Sasakian metric. Then there exists a smooth curve $\overline\Phi_t$ of $G$-invariant extremal $K$-contact metrics with $\Phi_0=\overline\Phi_0$ and $\overline\Phi_t$ diffeomorphic to $\Phi_t$.
\end{thm}
\begin{proof}
We follow mainly Boyer--Galicki--Simanca proof~\cite{bo-gal-sim}. However, in our case, $J_t=\Phi_t^\T$ may vary.
Let $\mathfrak{g}_\eta=\{\eta(X)\,|\,X\in \Lie(G)\}$ be the space of contact Hamiltonian functions associated to $\Lie(G)$.
Using Theorem~\ref{space_of_deformation}, we consider the deformations of $(\eta,\xi,\Phi_0,g_0)$ defined by
\begin{align*}
\eta_{t,\phi}&=\eta+\mathbb{G}_{t}d^\c_{t}\Delta_{B,t}\phi,\\
\Phi_{t,\phi}&=\Phi_t-\left(\xi\otimes(\eta_{t,\phi}-\eta)\right)\circ\Phi_t,\\
g_{t,\phi}&=d\eta_{t,\phi}\circ\left(\id\otimes\Phi_{t,\phi}\right)+\eta_{t,\phi}\otimes\eta_{t,\phi},
\end{align*}
where $\mathbb{G}_{t}$ is the Green's operator associated to the basic Laplacian $\Delta_{B,t}$, with respect to the $K$-contact metric $(\eta,\xi,\Phi_t,g_t),$
$\phi$ is an element of the space $C^{\infty,\perp}_G$ of smooth $G$-invariant basic functions which are $L^2$-orthogonal (with respect to $dv_\eta$) to $\mathfrak{g}_\eta$.
Here, $d^\c_t$ stands for $\Phi_td_B$. 

Denote by $\Pi_{\eta_{t,\phi}}$ the $L^2$-orthogonal projection of basic functions on
the space $\mathfrak{g}_{\eta_{t,\phi}}=\{\eta_{t,\phi}(X)\,|\,X\in \Lie(G)\}$
with respect to the volume form $dv_{\eta_{t,\phi}}$. Let $\mathcal{W}^{p,k}$ be the Sobolev completion
of $C^{\infty,\perp}_G$ involving derivatives up to order $k$.

Let $\mathcal{U}\subset\mathbb{R}\times\mathcal{W}^{p,k}$ be a neighborhood of $(0,0)$ such 
that $(\eta_{t,\phi},\xi,\Phi_{t,\phi},g_{t,\phi})$ is a $K$-contact structure for any $(t,\phi)\in\mathcal{U}$ and that
$\ker(\id-\Pi_\eta)\circ(\id-\Pi_{\eta_{t,\phi}})=\ker(\id-\Pi_{\eta_{t,\phi}})$
(by possibly shrinking $\mathcal{U}$). Consider then the map (defined by extension)
\begin{eqnarray*}
\Psi\colon \mathcal{U}\subset\mathbb{R}\times\mathcal{W}^{p,k+4}&\longrightarrow&\mathbb{R}\times\mathcal{W}^{p,k}\\
(t,\phi)&\longmapsto&\left(t,(\id-\Pi_\eta)\circ(\id-\Pi_{\eta_{t,\phi}}) \bar{s}^\T_{t,\phi}\right),
\end{eqnarray*}
where $\bar{s}^\T_{t,\phi}$ is the transverse Hermitian scalar curvature of $(\eta_{t,\phi},\xi,\Phi_{t,\phi},g_{t,\phi})$. The
map is well defined for $pk>5$.

By Proposition~\ref{extreme_cond}, $\Psi(t,\phi)=(t,0)$ if and only if $(\eta_{t,\phi},\xi,\Phi_{t,\phi},g_{t,\phi})$ is an extremal $K$-contact structure.
Hence, by hypothesis, $\Psi(0,0)=(0,0)$.

$\Psi$ is a $C^1$ map. Indeed, the dimension of the kernel of the basic Laplacian, with respect to the metric $(\eta,\xi,\Phi_t,g_t)$, applied on $1$-forms,
is equal to the dimension of $H^1(M,\mathbb{R})$ (see~\cite[Proposition 7.2.3]{bo-ga}) and so
the dimension of the kernel of $\Delta_{B,t}$ is independent of $t$. Thus, $\mathbb{G}_{t}$ is a $C^1$ map (see~\cite[Theorem 6.1]{kac-gmi}) and consequently $\Psi$ is.
The linearization of $\Psi$ at $(0,0)$ is given by (see~\cite[Proposition 7.3]{bo-gal-sim})
\begin{equation*}
\left(D\Psi\right)_{(0,0)}(t,\phi)=\left(t,t\left(\star\right)-2\left(\id-\Pi_\eta\right)\mathrm{L}_B^{g_0}(\phi)\right),
\end{equation*}
where $\mathrm{L}_B^{g_0}(\phi)=\delta^{g_0}_B\delta^{g_0}_B\left(D^{g^\T_0}d_B\phi\right)^{J_0,-}$ is a basic self-adjoint transversally
elliptic differential operator of order $4$ (here, $\left(\star\right)$ denotes some expression depending on $\left.\frac{d}{dt}\right|_{t=0}\Phi_t$).
By the standard arguments (see~\cite[Proposition 7.5]{bo-gal-sim} and~\cite[Lemma 4]{apo-cal-gau-ton})
and the main result of~\cite{kac}, $\left(D\Psi\right)_{(0,0)}$ is an isomorphism.

It follows from the Inverse Function Theorem for Banach spaces that there exists a neighborhood
$\mathcal{V}\subset\mathbb{R}\times\mathcal{W}^{p,k+4}$ of $(0,0)$ and  $\epsilon>0$ such that, for $|t|<\epsilon,$
$(\eta_{\Psi|_\mathcal{V}^{-1}(t,0)},\xi,\Phi_{\Psi|_\mathcal{V}^{-1}(t,0)},g_{\Psi|_\mathcal{V}^{-1}(t,0)})$
is an extremal $K$-contact metric.

By a standard bootstrapping argument (used for instance in~\cite{lej-2}), we get then a smooth family of $G$-invariant extremal $K$-contact metrics defined for a sufficiently small $t$.
Theorem~\ref{openess} now follows from {{Gray's Stability Theorem}}.
\end{proof}

\begin{rem}\label{condition_satisfied}
Suppose that at time $t=0$ we have $b_{B}^+=1$. Then, using~\cite[Theorem 6.1]{kac-gmi}, $b_{B}^+=1$ for small $|t|<\epsilon$.
Remark~\ref{upper-continuous} now implies that $\Phi_t$ is automatically semi-Sasakian for small values of $|t|$.
\end{rem}

\let\c=\oldc


{\small
\begin{thebibliography}{AV1}
\bibitem{apo-cal-gau-ton} V.~Apostolov, D.~M.~J.~Calderbank, P.~Gauduchon and C.~W.~T{\o}nnesen-Friedman,
\emph{Extremal K\"ahler metrics on projective bundles over a curve,} Adv. Math. {\textbf{227}} (2011), no. 6, 2385--2424.
\bibitem{apo-dra} V.~Apostolov and T.~Dr\v{a}ghici, 
\emph{The curvature and the integrability of almost-K\"ahler manifolds: a survey,} Fields Inst. Communications Series {\textbf{35}}, AMS (2003), 25--53.
\bibitem{boyer-torus} C.~P.~Boyer, \emph{Maximal Tori in Contactomorphism Groups,} Differential Geom. Appl. {\textbf{31}} (2013), no. 2, 190--216.
\bibitem{bo-ga} C.~P.~Boyer and K.~Galicki,
\emph{Sasaki geometry,} Oxford Mathematical Monographs. Oxford University Press, Oxford (2008).
\bibitem{bo-gal-sim} C.~P. Boyer, K.~Galicki and S.~R.~Simanca,
\emph{Canonical Sasaki metrics}, Comm. Math. Phys. \textbf{279} (2008), no. 3,705--733.
\bibitem{bo-gal-sim-1} \bysame,
\emph{The Sasaki cone and extremal Sasaki metrics,} Riemannian topology and geometric structures
on manifolds, 263--290, Progr. Math., 271, Birkh\"auser Boston, Boston, MA (2009).
\bibitem{cal-1} E.~Calabi, 
\emph{Extremal K\"ahler metrics,} in Seminar of Differential Geometry, S. T. Yau (eds),
Annals of Mathematics Studies {\textbf{102}}, Princeton University Press (1982), 259--290.
\bibitem{cal-2} \bysame,
\emph{Extremal K\"ahler metrics II}, Differential Geometry and Complex Analysis, 95--114, Springer, Berlin (1985).
\bibitem{dra-li-zha} T.~Dr\v{a}ghici, T.-J.~Li and W.~Zhang,
\emph{On the J-anti-invariant cohomology of almost complex 4-manifolds,} Q. J. Math. {\textbf {64}} (2013), no. 1, 83--111.
\bibitem {don} S.~K. Donaldson,
\emph{Remarks on Gauge theory, complex geometry and 4-manifold topology}, in ``The Fields Medallists Lectures" (eds. M. Atiyah and D. Iagolnitzer), pp. 384--403, World Scientific (1997).
\bibitem{fuj} A.~Fujiki, 
\emph{Moduli space of polarized algebraic manifolds and K\"{a}hler metrics}, Sugaku Expositions {\textbf {5}} (1992), 173--191.
\bibitem{fuj-sch} A.~Fujiki and G.~Schumacher,
\emph{The moduli space of extremal compact K\"ahler manifolds and generalized Weil--Petersson metrics,} Publ. Res. Inst. Math. Sci. {\textbf{26}} (1990), no. 1, 101--183.
\bibitem{fut} A.~Futaki,
\emph{An obstruction to the existence of Einstein K\"ahler metrics,} Invent. Math., \textbf{73} (1983), 437--443.
\bibitem{fut-mab} A.~Futaki and T.~Mabushi,
\emph{Bilinear forms and extremal K\"ahler vector fields associated with K\"ahler classes}, Math. Ann. {\textbf{301}} (1995), 199--210.
\bibitem{fu-on-wa} A.~Futaki, H.~Ono and G.~Wang,
\emph{Transverse K\"ahler geometry of Sasaki manifolds and toric Sasaki-Einstein manifolds,} arXiv:math/0607586.
\bibitem{gau} P.~Gauduchon,
\emph{Calabi's extremal K\"ahler metrics: An elementary introduction,} In preparation.
\bibitem{gau-1}\bysame, 
\emph{Hermitian connections and Dirac operators,} Boll. Un. Mat. Ital. B (7) {\textbf{11}} (1997), no. 2, suppl., 257--288.
\bibitem{geiges}
H.~Geiges,
\emph{An introduction to contact topology},
{Cambridge Studies in Advanced Mathematics}
\textbf{109},
{Cambridge University Press, Cambridge}
(2008).

\bibitem{gau-mar-spa-wal} J.~P.~Gauntlett, D.~Martelli, J.~Sparks, W.~Waldram, 
\emph{\ Sasaki-Einstein metrics on $S^2\times S^3$,} Adv. Theor. Math. Phys., 8 (2004), pp. 711--734.
\bibitem{gra} J.~W. Gray,
\emph{Some global properties of contact structures,} Ann. of Math. (2) \textbf{69} (1959), 421--450.
\bibitem{he} W. He,
\emph{On the transverse scalar curvature of a compact Sasaki manifold,} arXiv:1105.4000.
\bibitem{kac} A.~El.~Kacimi-Alaoui,
\emph{Op\'erateurs transversalement elliptiques sur un feuilletage riemannien et applications,}
Compositio Math., \textbf{79} (1990), 57--106.
\bibitem{kac-gmi} A.~El.~Kacimi-Alaoui and B.~Gmira,
\emph{Stabilit\'e du caract\`ere k\"ahl\'erien transverse,} Israel J. Math. {\textbf{101}} (1997), 323--347. 
\bibitem{kod-mor} K.~Kodaira and J.~Morrow,
\emph{Complex manifolds}, Holt, Rinehart and Winston (1971).
\bibitem{leb-sim} C.~LeBrun and S.~R.~Simanca,
\emph{On the K\"ahler Classes of Extremal Metrics}, Geometry and Global Analysis (Sendai, Japan 1993), FirstMath. Soc. Japan Intern. Res. Inst. Eds. Kotake, Nishikawa and Schoen.
\bibitem{lej-1} M.~Lejmi,
\emph{Extremal almost-K\"ahler metrics}, Internat. J. Math. \textbf{21} (2010), no. 12, 1639--1662.
\bibitem{lej-2} \bysame,
\emph{Stability under deformations of extremal almost-K\"ahler metrics in dimension $4$}, Math. Res. Lett. {\textbf{17}} (2010), no. 4, 601--612.
\bibitem{lej} \bysame,
\emph{Stability under deformations of Hermite-Einstein almost-K\"ahler metrics,}
To appear at annales de l'institut Fourier, \textbf{64} (2014).
\bibitem{lib} P.~Libermann, 
\emph{Sur les connexions hermitiennes,} C. R. Acad. Sci. Paris {\textbf{239}} (1954). 1579--1581.
\bibitem{sim-1} S.~R.~Simanca, 
\emph{Canonical metrics on compact almost complex manifolds}, Publica{\c{c}}{\~{o}}es Matem\'{a}ticas do IMPA, IMPA, Rio de Janeiro  (2004), 97 pp.
\bibitem{sim} \bysame,
\emph{Heat Flows for Extremal K\"ahler Metrics}, Ann. Scuola Norm. Sup. Pisa CL. Sci., \textbf{4} (2005), 187--217.
\bibitem{spa} J.~Sparks,
\emph{Sasakian-Einstein manifolds}, arXiv:1004.2461.

\bibitem{tondeur}
P.~Tondeur,
\newblock {\em Geometry of {R}iemannian foliations}, volume~20 of {\em Seminar
  on Mathematical Sciences}.
\newblock Keio University, Department of Mathematics, Yokohama (1994).


\end{thebibliography}
}
\end{document}